\documentclass[a4paper,12pt]{article}
\usepackage{amsmath,amssymb,amsthm,graphics,latexsym,amsfonts}
\usepackage{fancyhdr}
\usepackage{color}
\usepackage{graphics}
\usepackage{epstopdf}
\usepackage{amssymb}
\usepackage{cite}
\usepackage{hyperref}
\usepackage{tikz}
\usepackage{cleveref}

\hypersetup{
    colorlinks=true,
    linkcolor=cyan,
    filecolor=cyan,
    urlcolor=cyan,
    citecolor=cyan,
}

\title{\Large On the $C_4$-isolation number of a graph}
\author{ {Xiaohua Wei\,$^\text{a}$, Gang Zhang\,$^\text{b}$, Biao Zhao\,$^\text{a}$} \vspace{2mm}\\
	\small  $^\text{a}$\,College of Mathematics and System Sciences, Xinjiang University,\\
	\small  Urumqi, Xinjiang 830046, P.R. China\\
	\small  $^\text{b}$\,School of Mathematical Sciences, Xiamen University,\\
	\small  Xiamen, Fujian 361005, P.R. China\\
}
\date{\small (E-mails: wei\_xhua@163.com, gzh\_ang@163.com, zhb\_xj@163.com)}

\newtheorem{theorem}{Theorem}[section]
\newtheorem{lemma}[theorem]{Lemma}

\newtheorem{conjecture}[theorem]{Conjecture}
\newtheorem{proposition}[theorem]{Proposition}

\newcounter{claimcount}
\def\claimformat{\Alph{claimcount}}

\usepackage{indentfirst}

\tikzstyle{vertex}=[circle, draw, inner sep=0pt, minimum size=6pt]

\begin{document}

\maketitle

\small \noindent{\bfseries Abstract} Let $C_k$ be the cycle of length $k$. For any graph $G$, a subset $D \subseteq V(G)$ is a $C_k$-isolating set of $G$ if the graph obtained from $G$ by deleting the closed neighbourhood of $D$ contains no $C_k$ as a subgraph. The $C_k$-isolation number of $G$, denoted by $\iota(G,C_k)$, is the cardinality of a smallest $C_k$-isolating set of $G$. Borg (2020) and Borg et al. (2022) proved that if $G \ncong C_3$ is a connected graph of order $n$ and size $m$, then $\iota(G,C_3) \leq \frac{n}{4}$ and $\iota(G,C_3) \leq \frac{m+1}{5}$. Very recently, Bartolo, Borg and Scicluna showed that if $G$ is a connected graph of order $n$ that is not one of the determined nine graphs, then $\iota(G,C_4) \leq \frac{n}{5}$. In this paper, we prove that if $G \ncong C_4$ is a connected graph of size $m$, then $\iota(G,C_4) \leq \frac{m+1}{6}$, and we characterize the graphs that attain the bound. Moreover, we conjecture that if $G \ncong C_k$ is a connected graph of size $m$, then $\iota(G,C_k) \leq \frac{m+1}{k+2}$.

\section {\large Introduction}

Any graph considered in this paper is simple and finite. Let $G$ be a such graph. As usual, we use $V(G)$,$E(G)$, $\delta(G)$ and $\Delta(G)$ to denote the vertex set, edge set, minimum degree and maximum degree of $G$, respectively. We say that $u$ is a {\it neighbor} of $v$ in $G$ if $uv \in E(G)$, while $v$ is also a neighbor of $u$ in $G$. For a vertex $v \in V(G)$, the set $N_G(v)=\{u: u \in V(G) {\text{~and~}} uv \in E(G)\}$ is called the {\it open neighborhood of $v$} in $G$, and the set $N_G[v]=\{v\}\cup N_{G}(v)$ is called the {\it closed neighborhood of $v$} in $G$. For a subset $S \subseteq V(G)$, the set $N_{G}(S)=\bigcup_{v \in S}N_{G}(v)\setminus S$ is called the {\it open neighborhood of $S$} in $G$, and the set $N_{G}[S]=\bigcup_{v \in S}N_{G}[v]$ is called the {\it closed neighborhood of $S$} in $G$. Moreover, for a subgraph $H \subseteq G$, the set $N_{G}(H)=N_{G}(V(H))$ is called the {\it open neighborhood of $H$} in $G$, and the set $N_{G}[H]=N_{G}[V(H)]$ is called the {\it closed neighborhood of $H$} in $G$. Unless otherwise labeled, we may omit the subscript $G$ from the notations above; for example, we abbreviate $N_G(v)$ to $N(v)$. Let $G[S]$ denote the subgraph of $G$ induced by $S$ and $G-S$ denote the subgraph of $G$ induced by $V\setminus S$. A {\it leaf} of $G$ is a vertex of $G$ with degree 1. A {\it tree} is an acyclic connected graph. For any integer $k \geq 1$, denote $[k]=\{1,2,\cdots,k\}$. The readers are referred to \cite{Bollobas1998,Bondy2008} for more notations and terminologies in graph theory that we may use but not explicitly define here.

Let $\mathcal{F}$ be a set of connected graphs. A subset $D \subseteq V(G)$ is an {\it $\mathcal{F}$-isolating set} of a graph $G$ if $G-N[D]$ does not contain members of $\mathcal{F}$ as subgraphs. The {\it $\mathcal{F}$-isolation number} of $G$, denoted by $\iota(G,\mathcal{F})$, is the cardinality of a smallest $\mathcal{F}$-isolating set of $G$. If $\mathcal{F}=\{F\}$, then an $\mathcal{F}$-isolating set of $G$ is simply called an {\it $F$-isolating set} of $G$, while the $\mathcal{F}$-isolation number is called the {\it $F$-isolation number} and denote $\iota(G,F)=\iota(G,\mathcal{F})$ simply.

The problem of $\mathcal{F}$-isolating sets in graphs was first studied  by Caro and Hansberg \cite{Caro2017}, and they showed that if $G \ncong C_5$ is a connected graph of order $n \geq 3$, then $\iota(G,P_2) \leq \frac{n}{3}$, and this bound is sharp. Zhang and Wu in \cite{Zhang2021} and \cite{Zhang2022} separately proved that if $G \notin \{P_3,C_3,C_6\}$ is a connected graph of order $n$, then $\iota(G,P_3) \leq \frac{2n}{7}$, and if $G \ncong C_7$ is a connected graph of order $n$, then $\iota(G,P_4) \leq \frac{n}{4}$, where both bounds are sharp. Recently, Chen and Xu \cite{Chen2023} proved that if $G \ncong C_8$ is a connected graph of order $n$, then $\iota(G,P_5) \leq \frac{2n}{9}$, and this bound is sharp. Based on these results above, it is natural for us to conjecture that if $G \ncong C_{k+3}$ for $k \geq 6$ is a connected graph of order $n$, then $\iota(G,P_k) \leq \frac{2n}{k+4}$. This bound obviously can be attained by the cycle $C_{k+4}$.

For the research on $C_k$-isolation in graphs, there are not too many results yet. Borg \cite{Borg2020} proved that if $G \ncong C_3$ is a connected graph of order $n$, then $\iota(G,\mathcal{C}) \leq \frac{n}{4}$, where $\mathcal{C}=\{C_k: k\geq 3\}$. It is easy to see that any $\mathcal{C}$-isolating set of $G$ is a $C_3$-isolating set of $G$, and thus, $\iota(G,C_3) \leq \frac{n}{4}$. Due to Borg's construction in \cite{Borg2020}, these two bounds are also sharp. Very Recently, Bartolo, Borg and Scicluna \cite{Bartolo2023+} proved that if $G$ is a connected graph of order $n$ that is not one of the determined nine graphs, then $\iota(G,C_4) \leq \frac{n}{5}$, and this bound is sharp.

\begin{theorem}(Borg \cite{Borg2020}).\label{th1.1}
	If $G \ncong C_3$ is a connected graph of order $n$, then $\iota(G,C_3) \leq \frac{n}{4}$. This bound is sharp.
\end{theorem}

\begin{theorem}(Bartolo et al. \cite{Bartolo2023+}).\label{th1.2}
	If $G$ is a connected graph of order $n$ that is not one of the nine exceptional graphs, then $\iota(G,C_4) \leq \frac{n}{5}$. This bound is sharp.
\end{theorem}

So far, there have been a few results in the topic of $\mathcal{F}$-isolation. Except $P_k$- and $C_k$-isolation, we refer the readers to \cite{Borg2023,Borg&Fenech2020,Borg&Kaemawichanurat2020,Borg&Kaemawichanurat2023,Favaron2021,Kaemawichanurat2023,Huang2023+,Yan2022} for more types of isolation for certain subgraph structures, such as $K_k$-isolation, $K_{1,k}$-isolation and so on.

\vspace{3mm}
For an integer $k \geq 3$, let $C_k^+$ denote the graph obtained from a $C_k$ by adding a pendant edge to a vertex of it. A set of graphs $\mathcal{G}_k$ is recursively defined as follows:

(i) $C_k^+ \in \mathcal{G}_k$;

(ii) if $G' \in \mathcal{G}_k$, then $G \in \mathcal{G}_k$, where $G$ is the graph obtained from $G'$ by adding an edge which joins the leaf of a new $C_k^+$ to a vertex not on a $C_k$ of $G'$, or further,

(ii') if $G',G'' \in \mathcal{G}_k$, then $G \in \mathcal{G}_k$, where $G$ is the graph obtained from $G'$ and $G''$ by adding an edge which joins a vertex not on a $C_k$ of $G'$ and a vertex not on a $C_k$ of $G''$.

In practice, there is an equivalent construction of $\mathcal{G}_k$. Let ${\rm cons}(T,C_k)$ be the graph obtained from a tree $T$ and $V(T)$ copies of $C_k$ (vertex-disjoint of each other) such that each vertex of $T$ is joined to a vertex of a $C_k$ by one edge. We can easily see that $\mathcal{G}_k=\{{\rm cons}(T,C_k): T \text{~is a tree}\}$ also.

\vspace{3mm}
Borg, Fenech and Kaemawichanurat \cite{Borg&Fenech2022} proved that if $G$ is a connected graph of size $m$ that is not a $K_k$, then $\iota(G,K_k) \leq \frac{m+1}{\binom{k}{2}+2}$. Moreover, they characterized the graphs that attain the upper bound. As an immediate consequence of their result for the case of $k=3$, the following theorem holds.

\begin{theorem}(Borg et al. \cite{Borg&Fenech2022}).\label{th1.3}
	If $G \ncong C_3$ is a connected graph of size $m$, then $\iota(G,C_3) \leq \frac{m+1}{5}$, with equality if and only if $G \in \mathcal{G}_3$.
\end{theorem}

We use $K_4^-$ to denote the {\it diamond graph} that is obtained from a $K_4$ by deleting an edge. Inspired by the aforementioned work, we shall in this paper establish another upper bound on $\iota(G,C_4)$ of a graph $G$ in terms of its size.

\begin{theorem}\label{th1.4}
	If $G \ncong C_4$ is a connected graph of size $m$, then $$\iota(G,C_4) \leq \frac{m+1}{6}.$$ Moreover, the bound is attained if and only if $G \in \{K_4^-\} \cup \mathcal{G}_4$.
\end{theorem}

Combining the $\frac{m+1}{5}$'s bound in Theorem \ref{th1.3} and the result of this paper, we would like to propose the following conjecture.

\begin{conjecture}\label{conj1.5}
	If $G \ncong C_k$ is a connected graph of size $m$, then $$\iota(G,C_k) \leq \frac{m+1}{k+2}.$$
\end{conjecture}

Take a graph $G \in \mathcal{G}_k$. Then, $G={\rm cons}(T,C_k)$ for some tree $T$. For convenience, we call a vertex of $T$ a {\it connection vertex} of $G$, and an (induced) subgraph $C_k^+$ of $G$ a {\it $C_k^+$-constituent} of $G$. Clearly, each connection vertex of $G$ corresponds to a $C_k^+$-constituent of $G$, and $V(T)$ is the set of all connection vertices of $G$.


\begin{proposition}\label{prop1.6}
	If $G \in \mathcal{G}_k$ is a graph of size $m$, then $\iota(G,C_k)=\frac{m+1}{k+2}$.
\end{proposition}

\begin{proof}
	Suppose that $G={\rm cons}(T,C_4)$ for some tree $T$. Then, $m=6|V(T)|-1$. It is clear that $V(T)$ is a $C_k$-isolating set of $G$, and $\iota(G,C_4) \leq |V(T)|=\frac{m+1}{6}$. For another, let $D$ be a $C_k$-isolating set of $G$. It is easy to see that $|D \cap V(C_4^+)| \geq 1$ for each $C_k^+$-constituent of $G$, and $\iota(G,C_4) \geq |V(T)| = \frac{m+1}{6}$.
\end{proof}

By Proposition \ref{prop1.6}, we know that if Conjecture \ref{conj1.5} is true, then the upper bound is sharp. The extremal graph ${\rm cons}(K_{1,3}^+,C_4)$ is shown in Fig. \hyperlink{Fig1}{1} for $k=4$, where $K_{1,3}^+$ is the tree obtained from a $K_{1,3}$ by adding a pendant edge to a leaf of it. One can check that $\iota({\rm cons}(K_{1,3}^+,C_4),C_4) = 5 =\frac{29+1}{6}=\frac{|E({\rm cons}(K_{1,3}^+,C_4))|+1}{6}$.

\begin{center}
	\scalebox{1}[1]{\includegraphics{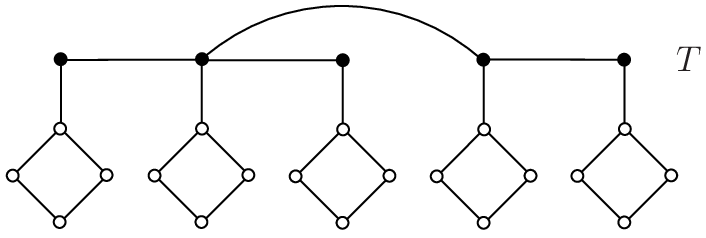}}
	\par \centerline{{\bf Fig. 1.} ~The graph $G={\rm cons}(K_{1,3}^+,C_4)$ with $m=29$ and $\iota(G,C_4)=5=\frac{m+1}{6}$.\hypertarget{Fig1}}
\end{center}

\section {\large Proof of Theorem \ref{th1.4}}

In this section, we prove Theorems \ref{th1.4}. For a graph $G=(V,E)$, let $E(V_1,V_2)=\{v_1v_2 \in E(G): v_1 \in V_1 \text{~and~} v_2 \in V_2\}$ and $e(V_1,V_2)=|E(V_1,V_2)|$, where $V_1,V_2 \subseteq V(G)$ and $V_1 \cap V_2 =\emptyset$. There are two lemmas used  frequently in the proof.

\begin{lemma}\label{lem2.1}
	Let $G=(V,E)$ be a graph and $S \subseteq V(G)$ a subset. If $G[S]$ has a $C_4$-isolating set $D$ such that $e(V(H),V \setminus S) \leq 1$ for each component $H$ of $G[S]-N[D]$, then $\iota(G,C_4) \leq |D|+\iota(G-S,C_4)$.
\end{lemma}

\begin{proof}
	Let $D'$ be a smallest $C_4$-isolating set of $G-S$. Then $|D'|=\iota(G-S,C_4)$. We shall prove that $D \cup D'$ is a $C_4$-isolating set of $G$, which implies that $\iota(G,C_4) \leq |D \cup D'|=|D|+|D'|=|D| + \iota(G-S,C_4)$. Clearly, each component $H$ of $G[S]-N[S]$ contains no $C_4$, and each component $H'$ of $(G-S)-N[D']$ contains no $C_4$. Since $e(V(H),V\setminus S) \leq 1$, $e(V(H),V(H')) \leq 1$. It is easy to see that the graph $H+H'$ must contain no $C_4$, where $V(H+H')=V(H) \cup V(H')$ and $E(H+H')=E(H) \cup E(H') \cup E(V(H),V(H'))$. Hence, $G-D \cup D'$ consists of some components, each of which contains no $C_4$ as a subgraph. The result is proved.
\end{proof}

\begin{lemma}(Zhang and Wu \cite{Zhang2021}).\label{lem2.2}
	If $G_1,G_2,\cdots,G_s$ are the distinct components of a graph $G$, then $\iota(G,C_4)=\sum_{i=1}^s\iota(G_i,C_4)$.
\end{lemma}

\noindent\textbf{Proof of Theorem \ref{th1.4}.} The sufficiency condition of the equality holds by Proposition \ref{prop1.6}. In what follows, we shall prove the bound and necessity in the theorem.

Let $G=(V,E)$ be a connected graph of order $n$ and size $m$. Suppose $G \ncong C_4$. Apply induction on $m$. If $G$ has no $C_4$, then $\iota(G,C_4)=0 < \frac{m+1}{6}$. Assume that $G$ has a $C_4$. Then $m \geq 4$. If $m \leq 5$, then $G \in \{K_4^-,C_4^+\}$. Clearly, $\iota(K_4^-,C_4)=\iota(C_4^+,C_4)=1$ and $\iota(G,C_4)=1=\frac{5+1}{6}=\frac{m+1}{6}$. Note that $C_4^+ \in \mathcal{G}_4$. The result holds. Assume that $m \geq 6$. If $\Delta=\Delta(G) \leq 2$, then $G \in \{P_n,C_n\}$ for $n \geq 6$. Clearly, $G$ has no $C_4$ now. We may assume that $\Delta \geq 3$.

\vspace{3mm}
\noindent{\bf Case 1:} $\Delta=3$.
\vspace{3mm}

Denote by the cycle $C=C_4:=u_1u_2u_3u_4u_1$. If $e(V(C),V\setminus V(C))=0$, then $G = G[V(C)]$. Since $m \geq 6$, $G \cong K_4$. Clearly, $\iota(G,C_4)=\iota(K_4,C_4)=1<\frac{6+1}{6}=\frac{m+1}{6}$. Assume that $e(V(C),V\setminus V(C)) \geq 1$. Due to $\Delta=3$, we know $G[V(C)] \ncong K_4$ and $G[V(C)] \in \{C_4,K_4^-\}$. Furthermore, $e(V(C),V\setminus V(C)) \leq 2$ if $G[V(C)] \cong K_4^-$, and $e(V(C),V\setminus V(C)) \leq 4$ if $G[V(C)] \cong C_4$.

\vspace{3mm}
\noindent{\bf Subcase 1.1:} $G[V(C)] \cong K_4^-$.
\vspace{3mm}

By the present assumption, we may assume that $d_{G[V(C)]}(u_i)=2$ for each $i \in \{1,3\}$ and $d_{G[V(C)]}(u_i)=d(u_i)=3$ for each $i \in \{2,4\}$. Let $G'=G-V(C)$. Denote $e=e(V(C),V(G'))=e(V(C),V\setminus V(C))$ simply. Hence, $1 \leq e \leq 2$. It is clear that $G' \ncong (\emptyset,\emptyset)$ is a graph with at most two components, and we denote them by $G_1'$ and $G_2'$ if there are two components of $G'$.

(i) We first consider that case that there is a component of $G'$ isomorphic to $C_4$. If $G' \cong C_4$, then $m=|E(G[V(C)])|+e(V(C),V(G'))+|E(G')|=|E(K_4^-)|+e+|E(C_4)|=5+e+4=9+e$. By $1 \leq e \leq 2$, $10 \leq m \leq 11$. Without loss of generality, we let $e(\{u_1\},V(G'))=1$. It is easy to check that $\{u_1\}$ is a $C_4$-isolating set of $G$, and thus, $\iota(G,C_4) \leq |\{u_1\}|=1 <\frac{10+1}{6} \leq \frac{m+1}{6}$.

If $G_i' \cong C_4$ for each $i \in [2]$, then $m=|E(G[V(C)])|+e(V(C),V(G'))+|E(G')|=|E(K_4^-)|+e+2|E(C_4)|=5+2+8=15$. It is easy to see that $e(\{u_1\},V(G'))=e(\{u_3\},V(G'))=1$, and $\{u_1,u_3\}$ is a $C_4$-isolating set of $G$. Thus, $\iota(G,C_4) \leq |\{u_1,u_3\}|=2<\frac{15+1}{6}=\frac{m+1}{6}$.

If there is exactly one $C_4$-component of two components of $G'$, then we may assume that $G_1' \cong C_4$ and $G_2' \ncong C_4$. Clearly, $G_2'=G-V(C)\cup V(G_1')$ and $|E(G_2')|=m-|E(G[V(C)])|-e(V(C),V(G'))-|E(G_1')|=m-5-2-4=m-11<m$. By the induction hypothesis, $\iota(G_2',C_4) \leq \frac{|E(G_2')|+1}{6}=\frac{m-11+1}{6}=\frac{m-10}{6}$. Let $e(\{u_1\},V(G_1'))=1$. Then, $\{u_1\}$ is a $C_4$-isolating set of $G[V(C)\cup V(G_1')]$. Note that $G[V(C)\cup V(G_1')]=G-V(G_2')$ and $e(V(C) \cup V(G_1'),V(G_2'))=e(\{u_3\},V(G_2'))=1$. By Lemma \ref{lem2.1}, we have $\iota(G,C_4) \leq |\{u_1\}|+\iota(G_2',C_4)\leq 1+\frac{m-10}{6}=\frac{m-4}{6}<\frac{m+1}{6}$.

(ii) It remains to consider the case that there is no component of $G'$ isomorphic to $C_4$. Let $c$ be the number of components of $G'$. Then, $1 \leq c \leq 2$ and $c \leq e$. Note that $|E(G')|=m-|E(G[V(C)])|-e(V(C),V(G')) = m-5-e<m$. By Lemma \ref{lem2.2} and the induction hypothesis, $\iota(G',C_4) \leq \frac{|E(G')|+c}{6}= \frac{m-5-e+c}{6}$, with equality if and only if each component of $G'$ is a member of $\{K_4^-\} \cup \mathcal{G}_4$. It is observed that $\{u_2\}$ is a $C_4$-isolating set of $G[V(C)]$, and $G[V(C)]-N[\{u_2\}] \cong (\emptyset,\emptyset)$. By Lemma \ref{lem2.1}, we have $\iota(G,C_4) \leq |\{u_2\}| +\iota(G',C_4) \leq 1+\frac{m-5-e+c}{6}=\frac{m+1}{6}+\frac{c-e}{6} \leq \frac{m+1}{6}$, and the equality holds if and only if $\iota(G,C_4)=1+\iota(G',C_4)$ and $c=e$, and $G' \in \{K_4^-\} \cup \mathcal{G}_4$ or $G'_i \in \{K_4^-\} \cup \mathcal{G}_4$ for each $i \in [2]$.

If $G' \in \{K_4^-\} \cup \mathcal{G}_4$ is connected, then $c=e=1$. Let $\{u_1v_1\} = E(V(C),V(G')) \subset E(G)$ for some $v_1 \in V(G')$. Assume that $G' \cong K_4^-$. Then, $m=|E(G[V(C)])|+e+|E(G')|=11$. It is easy to see that $\{v_1\}$ is a $C_4$-isolating set of $G$, and thus, $\iota(G,C_4) \leq |\{v_1\}|=1<\frac{11+1}{6}=\frac{m+1}{6}$. Assume that $G' \in \mathcal{G}_4$. Recall that $\mathcal{G}_4=\{{\rm cons}(T,C_4): T \text{~is a tree}\}$. Let $G'={\rm cons}(T',C_4)$ for some tree $T'$, and $v_1'$ be the connection vertex of the $C_4^+$-constituent of $G'$ containing the vertex $v_1$. It is easy to see that $m=|E(G[V(C)])|+e+|E(G')|=6+6|V(T')|-1=6|V(T')|+5$, and $\{v_1\}\cup V(T')\setminus \{v_1'\}$ is a $C_4$-isolating set of $G$. Thus, we have $\iota(G,C_4) \leq |\{v_1\}\cup V(T')\setminus \{v_1'\}|=|V(T')|<|V(T')|+1=\frac{6|V(T')|+5+1}{6}=\frac{m+1}{6}$.

If $G'$ is disconnected, then $c=e=2$ and $G_i' \in \{K_4^-\} \cup \mathcal{G}_4$ for each $i \in [2]$. Let $u_1v_1,u_3v_3 \in E(V(C),V(G')) \subset E(G)$ for some $v_1 \in V(G_1')$ and $v_3 \in V(G_2')$. By the symmetry of $G_1'$ and $G_2'$, there are three cases to consider: $G_i' \cong K_4^-$ for each $i$, $G_1' \cong K_4^-$ and $G_2' \in \mathcal{G}_4$, or $G_i' \in \mathcal{G}_4$ for each $i$. For the first case, we see that $m=5+2+5+5=17$, and $\{v_1,v_3\}$ is a $C_4$-isolating set of $G$. Thus, we have $\iota(G,C_4)\leq |\{v_1,v_3\}|=2<\frac{17+1}{6}=\frac{m+1}{6}$. For the second case, let $G_2'={\rm cons}(T_2',C_4)$ for some tree $T_2'$. We see that $m=5+2+5+6|V(T_2')|-1=6|V(T_2')|+11$, and $\{v_1\}\cup V(T_2')$ is a $C_4$-isolating set of $G$. Thus, we have $\iota(G,C_4) \leq |\{v_1\}\cup V(T_2')|=1+|V(T_2')|<|V(T_2')|+2=\frac{6|V(T_2')|+11+1}{6}=\frac{m+1}{6}$. For the third case, let $G_i'={\rm cons}(T_i',C_4)$ for some tree $T_i'$, where $i \in [2]$. Let $v_1'$ be the connection vertex of the $C_4^+$-constituent of $G_1'$ containing the vertex $v_1$. We see that $m=5+2+6|V(T_1')|-1+6|V(T_2')|-1=6(|V(T_1')|+|V(T_2')|)+5$, and $D=\{v_1\}\cup V(T_1')\setminus \{v_1'\} \cup V(T_2')$ is a $C_4$-isolating set of $G$. Thus, we have $\iota(G,C_4) \leq |D|=|V(T_1')|+|V(T_2')|<|V(T_1')|+|V(T_2')|+1=\frac{6(|V(T_1')|+|V(T_2')|)+5+1}{6}=\frac{m+1}{6}$.

\vspace{3mm}
\noindent{\bf Subcase 1.2:} $G[V(C)] \cong C_4$.
\vspace{3mm}

By the present assumption, it is clear that $d_{G[V(C)]}(u_i)=2$ for each $i \in [4]$. Recall that $\Delta =3$. Denote $e=e(V(C),V\setminus V(C))$. Then, $1 \leq e \leq 4$.

\vspace{3mm}
\noindent{\bf Subcase 1.2.1:} $e(V(C),V\setminus V(C))=1$.
\vspace{3mm}

Let $\{u_1v\} = E(V(C),V\setminus V(C))\subset E(G)$ for some $v \in V\setminus V(C)$. Note that $G[\{v\}\cup V(C)] \cong C_4^+$. Let $G'=G-\{v\}\cup V(C)$. Denote $e' =e(\{v\}\cup V(C),V(G'))$. Since $m \geq 6$, $1 \leq e' \leq 2$. Hence, $G' \ncong (\emptyset,\emptyset)$ is a graph with at most two components, and we denote them by $G_1'$ and $G_2'$ if two components of $G'$ exist.

(i) Assume that there is a component of $G'$ isomorphic to $C_4$. If $G' \cong C_4$, then $m=|E(G[\{v\}\cup V(C)])|+e(\{v\}\cup V(C),V(G'))+|E(G')|=|E(C_4^+)|+e'+|E(C_4)|=5+e'+4=9+e'$. By $1 \leq e' \leq 2$, $10 \leq m \leq 11$.  It is easy to see that $\{v\}$ is a $C_4$-isolating set of $G$, and thus, $\iota(G,C_4) \leq |\{v\}|=1 <\frac{10+1}{6} \leq \frac{m+1}{6}$.

If $G_i' \cong C_4$ for each $i \in [2]$, then $m=|E(G[\{v\}\cup V(C)])|+e(\{v\}\cup V(C),V(G'))+|E(G')|=|E(C_4^+)|+e'+2|E(C_4)|=5+2+8=15$. It is easy to see that $\{v\}$ is a $C_4$-isolating set of $G$, and thus, $\iota(G,C_4) \leq |\{v\}|=1<\frac{15+1}{6}=\frac{m+1}{6}$.

If $G_1' \cong C_4$ and $G_2' \ncong C_4$ (by the symmetry of $G_1'$ and $G_2'$), then $G_2'=G-S$ with $S=\{v\}\cup V(C)\cup V(G_1')$, and $|E(G_2')|=m-|E(G[\{v\}\cup V(C)])|-e(\{v\}\cup V(C),V(G'))-|E(G_1')|=m-5-2-4=m-11<m$. By the induction hypothesis, $\iota(G_2',C_4) \leq \frac{|E(G_2')|+1}{6}=\frac{m-11+1}{6}=\frac{m-10}{6}$. It is easy to check that $\{v\}$ is a $C_4$-isolating set of $G[S]$, and $E(S\setminus N[v],V\setminus S)=\emptyset$. By Lemma \ref{lem2.1}, we have $\iota(G,C_4) \leq |\{v\}|+\iota(G_2',C_4)\leq 1+\frac{m-10}{6}=\frac{m-4}{6}<\frac{m+1}{6}$.

(ii) Assume that each component of $G'$ is not isomorphic to $C_4$. Let $c$ be the number of components of $G'$. Then, $1 \leq c \leq 2$ and $c \leq e'$. Note that $|E(G')|=m-|E(G[\{v\}\cup V(C)])|-e(\{v\}\cup V(C),V(G')) = m-5-e'<m$. By Lemma \ref{lem2.2} and the induction hypothesis, $\iota(G',C_4) \leq \frac{|E(G')|+c}{6}= \frac{m-5-e'+c}{6}$, with equality if and only if each component of $G'$ is a member of $\{K_4^-\} \cup \mathcal{G}_4$. It is observed that $\{u_1\}$ is a $C_4$-isolating set of $G[\{v\}\cup V(C)]$, and $e(\{u_3\},V(G'))=0$. By Lemma \ref{lem2.1}, we have $\iota(G,C_4) \leq |\{u_1\}| +\iota(G',C_4) \leq 1+\frac{m-5-e'+c}{6}=\frac{m+1}{6}+\frac{c-e'}{6} \leq \frac{m+1}{6}$, and the equality holds if and only if $\iota(G,C_4)=1+\iota(G',C_4)$ and $c=e'$, and $G' \in \{K_4^-\} \cup \mathcal{G}_4$ or $G'_i \in \{K_4^-\} \cup \mathcal{G}_4$ for each $i \in [2]$.

If $G' \cong K_4^-$ or $G_i' \cong K_4^-$ for some $i \in [2]$, then $G'$ or $G_i'$ is an induced subgraph of $G$ with $e(V(G'),V\setminus V(G'))=e'=c=1$ or $e(V(G_i'),V\setminus V(G_i'))=e(\{v\},V(G_i'))=1$, and these two cases are covered by Subcase 1.1. Thus, we are done. It follows that $G' \in \mathcal{G}_4$ or $G'_i \in \mathcal{G}_4$ for each $i \in [2]$, as well as $c=e'$.

If $G' \in \mathcal{G}_4$, then $c=e'=1$. Let $G'={\rm cons}(T',C_4)$ for some tree $T'$, and let $\{vv_1\} = E(\{v\},V(G'))$ for some $v_1 \in V(G')$. Then, $m=|E(G[\{v\}\cup V(C)])|+e'+|E(G')|=6+6|V(T')|-1=6|V(T')|+5$. If $v_1 \notin V(T')$ is not a connection vertex of $G'$, then let $v_1'$ be the connection vertex of the $C_4^+$-constituent of $G'$ containing $v_1$. It is easy to see that $\{v\}\cup V(T')\setminus \{v_1'\}$ is a $C_4$-isolating set of $G$, and thus, we have $\iota(G,C_4) \leq |\{v_1\}\cup V(T')\setminus \{v_1'\}|=|V(T')|<|V(T')|+1=\frac{6|V(T')|+5+1}{6}=\frac{m+1}{6}$. If $v_1 \in V(T')$ is a connection vertex of $G'$, then by the present structure of $G$ and the definition of $\mathcal{G}_4$, $G \in \mathcal{G}_4$, where $v$ is a connection vertex of $G$ and $G[\{v\}\cup V(C)]$ is the $C_4^+$-constituent of $G$ containing $v$. Furthermore, $G={\rm cons}(T,C_4)$, where $T$ is the tree with $V(T)=\{v\} \cup V(T')$ and $E(T)=\{vv_1\} \cup E(T')$.

If $G_i' \in \mathcal{G}_4$ for each $i \in [2]$, then $c=e'=2$. Let $G_i'={\rm cons}(T_i',C_4)$ for some tree $T_i'$, and let $\{vv_i\} = E(\{v\},V(G_i'))$ for some $v_i \in V(G_i')$. Then, $m=|E(G[\{v\}\cup V(C)])|+e'+|E(G_1')|+|E(G_2')|=7+6|V(T_1')|-1+6|V(T_2')|-1=6(|V(T_1')|+|V(T_2')|)+5$. Let $v_i'$ be the connection vertex of the $C_4^+$-constituent of $G_i'$ containing $v_i$. Note that $v_i=v_i'$ is possible for each $i \in [2]$. If there exists some $i \in [2]$ such that $v_i$ is not a connection vertex of $G_i'$, then it can be check that $D=\{v\}\cup V(T_1')\setminus \{v_1'\} \cup V(T_2')\setminus \{v_2'\}$ is a $C_4$-isolating set of $G$. Thus, we have $\iota(G,C_4) \leq |D| \leq |V(T_1')|+|V(T_2')|<|V(T_1')|+|V(T_2')|+1=\frac{6(|V(T_1')|+|V(T_2')|)+5+1}{6}=\frac{m+1}{6}$. If $v_i=v_i'$ for each $i \in [2]$, then by the present structure of $G$ and the definition of $\mathcal{G}_4$, $G \in \mathcal{G}_4$, where $v$ is a connection vertex of $G$ and $G[\{v\}\cup V(C)]$ is the $C_4^+$-constituent of $G$ containing $v$. Furthermore, $G={\rm cons}(T,C_4)$, where $T$ is the tree with $V(T)=\{v\} \cup V(T_1')\cup V(T_2')$ and $E(T)=\{vv_1,vv_2\} \cup E(T_1')\cup E(T_2')$.

\vspace{3mm}
\noindent{\bf Subcase 1.2.2:} $e(V(C),V\setminus V(C))=2$.
\vspace{3mm}

Recall that $G[V(C)] \cong C_4$. Let $G'=G-V(C)$. Then, $e=e(V(C),V(G'))=e(V(C),V\setminus V(C))=2$. Hence, $G' \ncong (\emptyset,\emptyset)$ is a graph with at most two components, and we denote them by $G_1'$ and $G_2'$ if two components of $G'$ exist. If there is a $C_4$-component of $G'$, then any case is covered by Subcase 1.2.1 except that $e=2$ and $G' \cong C_4$ is connected. The reason is that, $G'$ or $G_i'$ for some $i \in [2]$ is an induced $C_4$-subgraph of $G$ with $e(V(G'),V\setminus V(G'))=e=1$ or $e(V(G_i'),V\setminus V(G_i'))=e(V(G_i'),V(C))=1$. For the exceptional case, we know that $m=|E(G[V(C)])|+e(V(C),V(G'))+|E(G')|=2|E(C_4)|+e=10$. Without loss of generality, we let $u_1v_1 \in E(V(C),V(G'))$ for some $v_1 \in V(G')$. It is clear that $\{u_1\}$ is a $C_4$-isolating set of $G$, and thus, $\iota(G,C_4) \leq |\{u_1\}|=1<\frac{10+1}{6}=\frac{m+1}{6}$.

If there is a $K_4^-$-component of $G'$, then $G' \cong K_4^-$ is connected, or $G'$ is disconnected and $G_i' \cong K_4^-$ for some $i \in [2]$. Hence, $G'$ or $G_i'$ is an induced $K_4^-$-subgraph of $G$ with $e(V(G'),V\setminus V(G'))=e=2$ or $e(V(G_i'),V\setminus V(G_i'))=e(V(G_i'),V(C))=1$. These two cases are covered by Subcase 1.1, and we are done.

(i) Assume that there is a component of $G'$ isomorphic to a member of $\mathcal{G}_4$. If $G' \in \mathcal{G}_4$ is connected, then let $G'={\rm cons}(T',C_4)$ for some tree $T'$, and we may let $u_1v \in E(V(C),V(G'))$ for some $v \in V(G')$. Note that $m=|E(G[V(C)])|+e+E(G')=6+6|V(T')|-1=6|V(T')|+5$. Let $v'$ be the connection vertex of the $C_4^+$-constituent of $G'$ containing $v$. It is easy to check that $D=\{v\}\cup V(T')\setminus \{v'\}$ is a $C_4$-isolating set of $G$, and thus, $\iota(G,C_4) \leq |D|=|V(T')|<|V(T')|+1= \frac{6|V(T')|+5+1}{6}=\frac{m+1}{6}$.

If $G'$ is disconnected and $G_1' \in \mathcal{G}_4$ (by the symmetry of $G_1'$ and $G_2'$), then let $G_1'={\rm cons}(T_1',C_4)$ for some tree $T_1'$, and we may let $\{u_1v_1\} = E(V(C),V(G_1'))$ for some $v_1 \in V(G_1')$. Clearly, $G_2'=G-V(C)\cup V(G_1')$ and $|E(G_2')|=m-|E(G[V(C)])|-e-|E(G_1')|=m-6-(6|V(T_1')|-1)=m-6|V(T_1')|-5$. By the induction hypothesis, $\iota(G_2',C_4) \leq \frac{|E(G_2')|+1}{6}=\frac{m-6|V(T_1')|-4}{6}$. Let $v_1'$ be the connection vertex of the $C_4^+$-constituent of $G_1'$ containing $v_1$. It is easy to check that $\{v_1\}\cup V(T_1')\setminus \{v_1'\}$ is a $C_4$-isolating set of $G[V(C)\cup V(G_1')]$. Note that $e(V(C)\cup V(G_1'),V(G_2')) =e(V(C),V(G_2'))=1$. By Lemma \ref{lem2.1}, we have $\iota(G,C_4) \leq |\{v_1\}\cup V(T_1')\setminus \{v_1'\}|+\iota(G_2',C_4)\leq |V(T_1')|+\frac{m-6|V(T_1')|-4}{6}=\frac{m-4}{6}<\frac{m+1}{6}$.

(ii) Assume that each component of $G'$ is not a member of $\{C_4,K_4^-\}\cup \mathcal{G}_4$. Note that $G'$ may still be disconnected, and that $|E(G')|=m-|E(G[V(C)])|-e = m-6<m$. By Lemma \ref{lem2.2} and the induction hypothesis, $\iota(G',C_4) \leq \frac{|E(G')|}{6}= \frac{m-6}{6}$. Since $\Delta=3$ and $e=2$, there is a vertex $u_i \in V(C)$ for $i \in [4]$ such that $e(V(C)\setminus N[u_i],V(G'))=0$. Consequently, $\{u_i\}$ is a $C_4$-isolating set of $G[V(C)]$ satisfying the condition of Lemma \ref{lem2.1}, and we have $\iota(G,C_4) \leq |\{u_i\}| +\iota(G',C_4) \leq 1+\frac{m-6}{6}=\frac{m}{6}<\frac{m+1}{6}$.

\vspace{3mm}
\noindent{\bf Subcase 1.2.3:} $e(V(C),V\setminus V(C))=3$.
\vspace{3mm}

Let $G'=G-V(C)$. Then, $e=e(V(C),V(G'))=e(V(C),V\setminus V(C))=3$. Hence, $G' \ncong (\emptyset,\emptyset)$ is a graph with at most three components. If there is a $C_4$-component of $G'$, then any case is covered by Subcases 1.2.1 and 1.2.2 except that $e=3$ and $G' \cong C_4$ is connected, for the analogous reason in the first paragraph of Subcase 1.2.2. For the exceptional case, we know that $m=|E(G[V(C)])|+e(V(C),V(G'))+|E(G')|=2|E(C_4)|+e=11$. Without loss of generality, we let $u_1v_1 \in E(V(C),V(G'))$ for some $v_1 \in V(G')$. Since $\Delta=3$, $e(\{u_3\},V(G')) \leq 1$. It is easy to see that $G-N[\{u_1\}]$ contains no $C_4$ and $\{u_1\}$ is a $C_4$-isolating set of $G$, which implies that $\iota(G,C_4) \leq |\{u_1\}|=1<\frac{11+1}{6}=\frac{m+1}{6}$.

If there is a $K_4^-$-component of $G'$, then $G'$ is disconnected. Moreover, there is a $K_4^-$-component $G_i'$ of $G'$ for some $i \in [2]$ or $i \in [3]$. Hence, $G_i'$ is an induced $K_4^-$-subgraph of $G$ with $e(V(G_i'),V\setminus V(G_i'))=e(V(G_i'),V(C)) \leq 2$. This case is covered by Subcase 1.1, and we are done.

(i) Assume that there is a component of $G'$ isomorphic to a member of $\mathcal{G}_4$. If $G' \in \mathcal{G}_4$ is connected, then let $G'={\rm cons}(T',C_4)$ for some tree $T'$. Note that $m=|E(G[V(C)])|+e+E(G')=7+6|V(T')|-1=6|V(T')|+6$. It is easy to check that $D=\{u_1\}\cup V(T')$ is a $C_4$-isolating set of $G$, and $e(\{u_3\},V(G'-N[T'])) \leq 1$. Thus, we have $\iota(G,C_4) \leq |D|=1+|V(T')|=\frac{6|V(T')|+6}{6}< \frac{6|V(T')|+6+1}{6}=\frac{m+1}{6}$.

If $G'$ is disconnected and $G_1' \in \mathcal{G}_4$ is a component of $G'$ (by the symmetry among components of $G'$), then let $G_1'={\rm cons}(T_1',C_4)$ for some tree $T_1'$. Without loss of generality, we may let $\{u_1v_1\} \in E(V(C),V(G_1'))$ for some $v_1 \in V(G_1')$. Let $v_1'$ be the connection vertex of the $C_4^+$-constituent $(C')^+$ of $G_1'$ containing $v_1$. Clearly, $C' \cong C_4$. If $e(V(G_1'),V(C))=1$, then $C'$ of $(C')^+$ is an induced $C_4$-subgraph of $G$ with $e(V(C'),V\setminus V(C')) \leq e(\{v_1'\},V(C'))+e(V(C),V(G_1'))=1+1=2$. This case is covered by Subcases 1.2.1 and 1.2.2, and we are done. So, we let $e(V(G_1'),V(C)) = 2$. Since $e(V(C),V(G'))=3$ and $G'$ is disconnected, $G'$ consists of exactly two components. Denote by $G_2'$ the other component. Then, $G_2'=G-V(C)\cup V(G_1')$ and $e(V(G_2'),V(C))=1$. Moreover, $|E(G_2')|=m-|E(G[V(C)])|-e-|E(G_1')|=m-6|V(T_1')|-6<m$. By the induction hypothesis, $\iota(G_2',C_4) \leq \frac{|E(G_2')|+1}{6}=\frac{m-6|V(T_1')|-5}{6}$. It is easy to check that $\{v_1\}\cup V(T'_1) \setminus \{v_1'\}$ is a $C_4$-isolating set of $G[V(C)\cup V(G_1')]$. By Lemma \ref{lem2.1}, we have $\iota(G,C_4) \leq |\{v_1\}\cup V(T_1')\setminus \{v_1'\}|+\iota(G_2',C_4)\leq |V(T_1')|+\frac{m-6|V(T_1')|-5}{6}=\frac{m-5}{6}<\frac{m+1}{6}$.

(ii) Assume that each component of $G'$ is not a member of $\{C_4,K_4^-\}\cup \mathcal{G}_4$. Note that $|E(G')|=m-|E(G[V(C)])|-e = m-7<m$. By Lemma \ref{lem2.2} and the induction hypothesis, $\iota(G',C_4) \leq \frac{|E(G')|}{6}= \frac{m-7}{6}$. There is a vertex $u_i \in V(C)$ for $i \in [4]$ with $e(V(C)\setminus N[u_i],V(G'))=0$, and $\{u_i\}$ is a $C_4$-isolating set of $G[V(C)]$. By Lemma \ref{lem2.1}, we have $\iota(G,C_4) \leq |\{u_i\}| +\iota(G',C_4) \leq 1+\frac{m-7}{6}=\frac{m-1}{6}<\frac{m+1}{6}$.

\vspace{3mm}
\noindent{\bf Subcase 1.2.4:} $e(V(C),V\setminus V(C))=4$.
\vspace{3mm}

Let $G'=G-V(C)$. Then, $e=e(V(C),V(G'))=e(V(C),V\setminus V(C))=4$. Hence, $G' \ncong (\emptyset,\emptyset)$ is a graph with at most four components. In the similar manner of proving previous subcases, we shall know that each component of $G'$ is not a member of $\{C_4,K_4^-\}$ except the case that $e=4$ and $G' \cong C_4$ is connected. For the exceptional case, it is clear that $m=|E(G[V(C)])|+e(V(C),V(G'))+|E(G')|=2|E(C_4)|+e=12$. Without loss of generality, we let $u_1v_1 \in E(V(C),V(G'))$ for some $v_1 \in V(G')$. We can easily see that $\{u_1,v_1\}$ is a $C_4$-isolating set of $G$, which implies that $\iota(G,C_4) \leq |\{u_1,v_1\}|=2<\frac{12+1}{6}=\frac{m+1}{6}$.

(i) Assume that there is a component of $G'$ isomorphic to a member of $\mathcal{G}_4$. If $G' \in \mathcal{G}_4$ is connected, then the proof is similar to the proof in (i) of Subcase 1.2.3, except for a slight difference. That is, $\iota(G,C_4) \leq |D|=1+|V(T')|=\frac{6|V(T')|+7}{6}< \frac{6|V(T')|+7+1}{6}=\frac{m+1}{6}$, where $T'$ and $D$ is defined as in (i) of Subcase 1.2.3.

If $G'$ is disconnected and $G_1' \in \mathcal{G}_4$ is a component of $G'$ (by the symmetry among components of $G'$), then let $G_1'={\rm cons}(T_1',C_4)$ for some tree $T_1'$. Without loss of generality, we may let $\{u_1v_1\} \in E(V(C),V(G_1'))$ for some $v_1 \in V(G_1')$. Let $v_1'$ be the connection vertex of the $C_4^+$-constituent $(C')^+$ of $G_1'$ containing $v_1$. Clearly, $C' \cong C_4$. If $1 \leq e(V(G_1'),V(C))\leq 2$, then $C'$ of $(C')^+$ is an induced $C_4$-subgraph of $G$ with $e(V(C'),V\setminus V(C')) \leq e(\{v_1'\},V(C'))+e(V(C),V(G_1'))\leq 1+2=3$. This case is covered by Subcases from 1.2.1 to 1.2.3, and we are done. So, we let $e(V(G_1'),V(C)) = 3$. Since $e(V(C),V(G'))=4$ and $G'$ is disconnected, $G'$ consists of exactly two components. Denote by $G_2'$ the other component. Then, $G_2'=G-V(C)\cup V(G_1')$ and $e(V(G_2'),V(C))=1$. Moreover, $|E(G_2')|=m-|E(G[V(C)])|-e-|E(G_1')|=m-6|V(T_1')|-7<m$. By the induction hypothesis, $\iota(G_2',C_4) \leq \frac{|E(G_2')|+1}{6}=\frac{m-6|V(T_1')|-6}{6}$. It is easy to check that $\{u_1\}\cup V(T'_1)$ is a $C_4$-isolating set of $G[V(C)\cup V(G_1')]$. By Lemma \ref{lem2.1}, we have $\iota(G,C_4) \leq |\{u_1\}\cup V(T'_1)|+\iota(G_2',C_4)\leq 1+|V(T_1')|+\frac{m-6|V(T_1')|-6}{6}=\frac{m}{6}<\frac{m+1}{6}$.

(ii) Assume that each component of $G'$ is not a member of $\{C_4,K_4^-\}\cup \mathcal{G}_4$. Note that $|E(G')|=m-|E(G[V(C)])|-e = m-8<m$. By Lemma \ref{lem2.2} and the induction hypothesis, $\iota(G',C_4) \leq \frac{|E(G')|}{6}= \frac{m-8}{6}$. Since $\Delta=3$ and $e=4$, $d(u_i)=3$ for each $i \in [4]$. It is easy to see that, $e(V(C)\setminus N[u_i],V(G'))=1$, and $\{u_i\}$ is a $C_4$-isolating set of $G[V(C)]$. By Lemma \ref{lem2.1}, we have $\iota(G,C_4) \leq |\{u_i\}| +\iota(G',C_4) \leq 1+\frac{m-8}{6}=\frac{m-2}{6}<\frac{m+1}{6}$.

\vspace{3mm}
\noindent{\bf Case 2:} $\Delta \geq 4$.
\vspace{3mm}

Fix a vertex $v$ with $d(v)=\Delta$. Let $G'=G-N[v]$ and $e=e(N[v],V(G'))$. If $e=0$, then $G=G[N[v]]$. Clearly, $\{v\}$ is a $C_4$-isolating set of $G$. Recall that $m \geq 6$. Thus, $\iota(G,C_4) \leq |\{v\}|=1<\frac{6+1}{6}\leq \frac{m+1}{6}$. Assume that $e \geq 1$. Then, $|E(G')|=m-|E(G[N[v])|-e\leq m-\Delta-e \leq m-5<m$.

Let $\mathcal{H}$ be the set of components of $G'$, $\mathcal{H}_1$ the set of $C_4$-components of $G'$, $\mathcal{H}_2$ the set of $K_4^-$-components of $G'$, and $\mathcal{H}_3$ the set of components of $G'$ isomorphic to a member of $\mathcal{G}_4$. Let $\mathcal{H}_s= \bigcup_{i=1}^3 \mathcal{H}_i$ and $\mathcal{H}'=\mathcal{H}\setminus \mathcal{H}_s$. Set $c=|\mathcal{H}|$, $s=|\mathcal{H}_s|$, $c'=|\mathcal{H}'|$, and $c_i=|\mathcal{H}_i|$ for $i \in [3]$. Clearly, $s=c_1+c_2+c_3$ and $c=s+c'\geq 1$. By the induction hypothesis, $\iota(H,C_4) \leq \frac{|E(H)|+1}{6}$ for each special component $H \in \mathcal{H}_2\cup \mathcal{H}_3$, and $\iota(H,C_4) \leq \frac{|E(H)|}{6}$ for each non-special component $H \in \mathcal{H}'$.

If $\mathcal{H}_s =\emptyset$, then $\mathcal{H}=\mathcal{H}'$ and $\mathcal{H}' \neq \emptyset$. Note that $\{v\}$ is a $C_4$-isolating set of $G[N[v]]$. By Lemmas \ref{lem2.1} and \ref{lem2.2}, and the induction hypothesis, we have $\iota(G,C_4) \leq |\{v\}|+\iota(G',C_4)\leq 1+\frac{|E(G')|}{6} \leq 1+ \frac{m-\Delta-e}{6} \leq 1+ \frac{m-4-1}{6}=\frac{m+1}{6}$. The equality holds if and only if $\iota(G,C_4)=1+\iota(G',C_4)$, $\iota(G',C_4)=\frac{|E(G')|}{6}$, $|E(G')|=m-\Delta-e$, $\Delta=4$ and $e=1$. Hence, we know that $G[N[v]] \cong K_{1,4}$, and $G$ is the graph obtained from $G'$ and a $K_{1,4}$ by adding an edge which joins a leaf of the $K_{1,4}$ to a vertex of $G'$. Clearly, any $C_4$-isolating set of $G'$ is a $C_4$-isolating set of $G$, and we have $\iota(G,C_4) \leq \iota(G',C_4) = \frac{|E(G')|}{6}=\frac{m-5}{6}<\frac{m+1}{6}$.

It remains to consider the case of $\mathcal{H}_s \neq \emptyset$. We divide the following proof into two subcases in terms of the value of $e(V(H),N[v])$ where $H \in \mathcal{H}_s$.

\vspace{3mm}
\noindent{\bf Subcase 2.1:} For some $H \in \mathcal{H}_s$, $|N(H)| = 1$ or $1 \leq e(V(H),N[v])\leq 2$.
\vspace{3mm}

Suppose that there exists a component $H^* \in \mathcal{H}_s$ with $1 \leq e(V(H^*),N[v])\leq 2$. Without loss of generality, we may assume that $v_1 \in N(v)$ and $v_1 \in N(H^*)$. Let $\mathcal{H}_s^{v_1}$ be the set of components of $G'$ such that $N(H)=\{v_1\}$, or $v_1 \in N(H)$ and $1 \leq e(V(H),N[v])\leq 2$, where $H \in \mathcal{H}_s$. Clearly, $H^* \in \mathcal{H}_s^{v_1}$ and $\mathcal{H}_s^{v_1} \neq \emptyset$. Set $S=\{v_1\} \cup \bigcup_{H\in \mathcal{H}_{s}^{v_1}}V(H)$ and $G''=G-S$. Since $v \in V(G'')$, $G'' \ncong (\emptyset,\emptyset)$. Let $(\mathcal{H}')^{v_1}$ be the set of components of $G'$ with $N(H)=\{v_1\}$ for $H \in \mathcal{H}'$. As in Fig. \hyperlink{Fig2}{2}, $G''$ consists of components of $(\mathcal{H}')^{v_1}$ and the component $G_v$ containing $v$. Particularly, it is possible that $(\mathcal{H}')^{v_1} =\emptyset$ and $G''=G_v$. By $d(v)=\Delta \geq 4$, $d_{G_v}(v)\geq 3$. Hence, $G_v \ncong C_4$. By the induction hypothesis, $\iota(H,C_4) \leq \frac{|E(H)|}{6}$ for each $H \in (\mathcal{H}')^{v_1}$, and $\iota(G_v,C_4) \leq \frac{|E(G_v)|+1}{6}$ with equality if and only if $G_v \in \{K_4^-\} \cup \mathcal{G}_4$.

\begin{center}
	\scalebox{1}[1]{\includegraphics{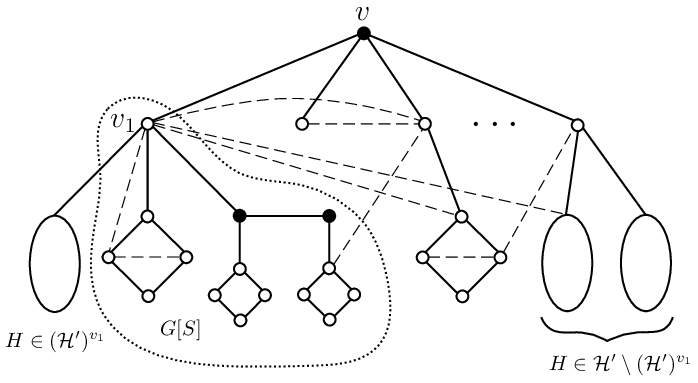}}
	\par \centerline{{\bf Fig. 2.} ~The set $S$ and graph $G''$ in Subcase 2.1.\hypertarget{Fig2}}
\end{center}

For each $i \in [3]$, let $\mathcal{H}_i^{v_1}=\mathcal{H}_i \cap \mathcal{H}_s^{v_1}$. Denote $c_i^{v_1}=|\mathcal{H}_{i}^{v_1}|$ and $(c')^{v_1}=|(\mathcal{H}')^{v_1}|$. We can see that $e(S,\bigcup_{H \in (\mathcal{H}')^{v_1}}V(H))=e(\{v_1\},\bigcup_{H \in (\mathcal{H}')^{v_1}}V(H)) \geq (c')^{v_1} \geq 0$ and $e(S,V(G_v)) = e(\{v_1\},V(G_v))+e(S \setminus \{v_1\}, N(v) \setminus \{v_1\}) \geq |\{vv_1\}|=1$. Assume that each component $H$ of $\mathcal{H}_3$ has $t(H)$ connection vertices. Then, $|E(H)|=6t(H)-1$ for each $H \in \mathcal{H}_3$. Moreover, since $\mathcal{H}_3^{v_1} \subseteq \mathcal{H}_3$, we obtain that $\sum_{H \in \mathcal{H}_3^{v_1}}|E(H)|=6t^{v_1}-c_3^{v_1}$, where $t^{v_1} = \sum_{H\in \mathcal{H}_3^{v_1}}t(H)$. Therefore, 
\begin{equation}
\begin{array}{lllll}
	|E(G'')|&=m-|E(G[S])|-e(S,V\setminus S)\\
		    &=m-\sum_{H \in \mathcal{H}_{s}^{v_1}}(|E(H)|+e(V(H),N[v]))\\
		    &~~~~~~-e(\{v_1\},\bigcup_{H \in (\mathcal{H}')^{v_1}}V(H))-e(\{v_1\},V(G_v))\\
		    &\leq m-5c_1^{v_1}-6c_2^{v_1}-(6t^{v_1}-c_3^{v_1})-c_3^{v_1}-(c')^{v_1}-|\{vv_1\}|\\
		    &=m-5c_1^{v_1}-6c_2^{v_1}-6t^{v_1}-(c')^{v_1}-1.
\end{array}
\end{equation}

For each $H \in \mathcal{H}_3^{v_1}$, let $D_H^{v_1}$ be the set of all connection vertices of $H$. Observe that $D_S=\{v_1\}\cup \bigcup_{H \in \mathcal{H}_3^{v_1}}D_H^{v_1}$ is a $C_4$-isolating set of $G[S]$, and each component of $G[S]-N[D_S]$ has at most one edge which joins a vertex of the component to a vertex of $N(v) \setminus \{v_1\}$. Note that $|E(G'')|=\sum_{H \in (\mathcal{H}')^{v_1}}|E(H)|+|E(G_v)|$ also. By (1), Lemmas \ref{lem2.1} and \ref{lem2.2}, and the induction hypothesis, we have
\begin{equation}
	\begin{array}{lllll}
		\iota(G,C_4)&\leq |D_S|+\iota(G'',C_4)\\
		&= |\{v_1\}\cup \bigcup_{H \in \mathcal{H}_3^{v_1}}D_H^{v_1}|+\sum_{H \in (\mathcal{H}')^{v_1}}\iota(H,C_4)+\iota(G_v,C_4)\\
		&\leq 1+t^{v_1}+\sum_{H \in (\mathcal{H}')^{v_1}}\frac{|E(H)|}{6}+\frac{|E(G_v)|+1}{6}\\
		&=1+t^{v_1}+\frac{|E(G'')|+1}{6}\\
		&\leq 1+t^{v_1}+\frac{m-5c_1^{v_1}-6c_2^{v_1}-6t^{v_1}-(c')^{v_1}-1+1}{6}\\
		&=\frac{m+1}{6}+\frac{5-5c_1^{v_1}-6c_2^{v_1}-(c')^{v_1}}{6}.
	\end{array}
\end{equation}
Assume that $5-5c_1^{v_1}-6c_2^{v_1}-(c')^{v_1} <0$. Then $\iota(G,C_4) < \frac{m+1}{6}$. So, it follows that $5-5c_1^{v_1}-6c_2^{v_1}-(c')^{v_1} \geq 0$, that is, $5c_1^{v_1}+6c_2^{v_1}+(c')^{v_1} \leq 5$. Further, $c_1^{v_1} \leq 1$.

(i) If $c_1^{v_1}=1$, then $c_2^{v_1}=(c')^{v_1}=0$. Hence, $G''=G_v$. Due to (2), $\iota(G,C_4) \leq \frac{m+1}{6}$. Note that the equality holds if and only if each equality throughout (1) and (2) holds. Consequently, we have $\iota(G'',C_4)=\iota(G_v,C_4)=\frac{|E(G_v)|+1}{6}$, and thus, $G'' \in \{K_4^-\}\cup \mathcal{G}_4$. Meanwhile, we have $|E(G'')|=m-5c_1^{v_1}-(6t^{v_1}-c_3^{v_1})-c_3^{v_1}-|\{vv_1\}|$, and thus, $e(\{v_1\},V(G''))=\{vv_1\}$ and $e(V(H),N[v])=e(\{v_1\},V(H))=1$ for each $H \in \mathcal{H}_s^{v_1}=\{H_1^{v_1}\}\cup \mathcal{H}_3^{v_1}$, where $\{H_1^{v_1}\}=\mathcal{H}_1^{v_1}$. Therefore, we can see that $G$ is the graph obtained from $H_1^{v_1}$, $G''$ and the components of $\mathcal{H}_3^{v_1}$ by adding the vertex $v_1$ to them and joining $v_1$ to a vertex of each one of them.

Recall that $H_1^{v_1} \cong C_4$ and $H \in \mathcal{G}_4$ for each $H \in \mathcal{H}_3^{v_1}$. Assume that $G'' \cong K_4^-$.  Clearly, $D_S=\{v_1\} \cup \bigcup_{H \in \mathcal{H}_3^{v_1}}D_H^{v_1}$ is a $C_4$-isolating set of $G$. Now $m=5+6+6t^{v_1}$, and we have $\iota(G,C_4) \leq |D_S|=1+t^{v_1}<t^{v_1}+2 =\frac{5+6+6t^{v_1}+1}{6}=\frac{m+1}{6}$. Assume that $G'' \in \mathcal{G}_4$. By the present structure of $G$, each component of $\mathcal{H}_3^{v_1}$ is symmetric with $G''$. If there is a component $H' \in \{G''\} \cup \mathcal{H}_3^{v_1}$ with $\{v_1u\} = E(\{v_1\},V(H'))$ for some $u \in V(H')$ such that $u$ is not a connection vertex of $H'$, then let $u'$ be the connection vertex of the $C_4^+$-constituent of $H'$ containing $u$. In particular, $u=v$ if $G''$ is the component $H'$. Let $D_{G''}^{v_1}$ be the set of all connection vertices of $G''$. It can be checked that $D=\{v_1\} \cup D_{H'}^{v_1}\setminus \{u'\} \cup \bigcup_{H\in \{G''\}\cup \mathcal{H}_3^{v_1}\setminus \{H'\}}D_H^{v_1}$ is a $C_4$-isolating set of $G$. Now $|E(G'')|=6|D_{G''}^{v_1}|-1$ and $m=5+6|D_{G''}^{v_1}|-1+1+6t^{v_1}=5+6|D_{G''}^{v_1}|+6t^{v_1}$, and we have $\iota(G,C_4) \leq |D|=|D_{G''}^{v_1}|+t^{v_1}<|D_{G''}^{v_1}|+t^{v_1}+1 =\frac{5+6|D_{G''}^{v_1}|+6t^{v_1}+1}{6}=\frac{m+1}{6}$. Otherwise, for each component $H \in \{G''\} \cup \mathcal{H}_3^{v_1}$ with $\{v_1u\} = E(\{v_1\},V(H))$ and $u \in V(H)$, $u$ is a connection vertex of $H$. Then, $G \in \mathcal{G}_4$ by the definition of $\mathcal{G}_4$, where $v_1$ is a connection vertex of $G$ and $G[\{v_1\}\cup V(H_1^{v_1})]$ is the $C_4^+$-constituent of $G$ containing $v_1$.

(ii) If $c_1^{v_1}=0$, then $c_2^{v_1}=0$. Since $\mathcal{H}_s^{v_1} \neq \emptyset$, $c_3^{v_1} \geq 1$. We shall prove this case by slightly modifying the set $D_S$ before the inequality (2). For each $H\in \mathcal{H}_3^{v_1}=\mathcal{H}_s^{v_1}$, let $v_1u_H^{v_1} \in E(\{v_1\},V(H))\subseteq E(V(H),N[v])$ for some $u_{H}^{v_1} \in V(H)$, and $(u_H^{v_1})'$ be the connection vertex of the $C_4^+$-constituent of $H$ containing $u_H^{v_1}$. Let $D_S'=\bigcup_{H\in \mathcal{H}_3^{v_1}}\left(\{u_H^{v_1}\}\cup D_{H}^{v_1}\setminus \{(u_H^{v_1})'\}\right)$. It is easy to check that $D_S'$ is a $C_4$-isolating set of $G[S]$, and each component of $G[S]-N[D_S]$ has at most one edge which joins a vertex of the component to a vertex of $N(v) \setminus \{v_1\}$. By (1), $c_1^{v_1}=c_2^{v_2}=0$, Lemmas \ref{lem2.1} and \ref{lem2.2}, and the induction hypothesis, we have
\begin{equation*}
	\begin{array}{lllll}
		\iota(G,C_4)&\leq |D_S'|+\iota(G'',C_4)\\
		&= |\bigcup_{H\in \mathcal{H}_3^{v_1}}\left(\{u_H^{v_1}\}\cup D_{H}^{v_1}\setminus \{(u_H^{v_1})'\}\right)|+\sum_{H \in (\mathcal{H}')^{v_1}}\iota(H,C_4)+\iota(G_v,C_4)\\
		&\leq t^{v_1}+\sum_{H \in (\mathcal{H}')^{v_1}}\frac{|E(H)|}{6}+\frac{|E(G_v)|+1}{6}\\
		&=t^{v_1}+\frac{|E(G'')|+1}{6}\\
		&\leq t^{v_1}+\frac{m-5c_1^{v_1}-6c_2^{v_1}-6t^{v_1}-(c')^{v_1}-1+1}{6}\\
		&=\frac{m+1}{6}+\frac{-1-(c')^{v_1}}{6}<\frac{m+1}{6}.
	\end{array}
\end{equation*}

\vspace{3mm}
\noindent{\bf Subcase 2.2:} For each $H \in \mathcal{H}_s$, $|N(H)| \geq 2$ and $e(V(H),N[v]) \geq 3$.
\vspace{3mm}

Recall that $\mathcal{H}_s \neq \emptyset$. By the present assumption, we set $S=N[v] \cup \bigcup_{H \in \mathcal{H}_s}V(H)$ and $G''=G-S$. Then, $G''$ consists of components of $\mathcal{H}'$. (i) If $\mathcal{H}' = \emptyset$, then $G$ is the graph obtained from $G[N[v]]$ and the components of $\mathcal{H}_s$ such that each component of $\mathcal{H}_s$ is joined to at least two vertices of $N(v)$ by at least three edges. Since $G'=G-N[v]$, $e=e(N[v],V(G')) \geq 3|\mathcal{H}_s|=3s=3(c_1+c_2+c_3)$. For each $H \in \mathcal{H}_3$, let $D_H$ be the set of all connection vertices of $H$, and let $t(H)=|D_H|$. Clearly, $\sum_{H \in \mathcal{H}_{3}}|E(H)|=6t-c_3$, where $t=\sum_{H \in \mathcal{H}_{3}}t(H)$. Hence, $m=|E(G[N[v]])|+e(N[v],V(G'))+\sum_{H \in \mathcal{H}_{s}}|E(H)| \geq \Delta+3(c_1+c_2+c_3)+4c_1+5c_2+6t-c_3=\Delta+7c_1+8c_2+2c_3+6t$, where $\Delta \geq 4$, $t \geq 0$ and $c_1+c_2+c_3 \geq 1$.

For each $H \in \mathcal{H}_1 \cup \mathcal{H}_2$, let $v_{H}u_{H} \in E(N(v),V(H))$ for some $v_{H} \in N(v)$ and $u_{H} \in V(H)$. It is observed that $D=\{v\}\cup \bigcup_{H \in \mathcal{H}_1 \cup \mathcal{H}_2}\{u_{H}\}\cup \bigcup_{H \in \mathcal{H}_3}D_H$ is a $C_4$-isolating set of $G$. Therefore, we have $\iota(G,C_4) \leq |D|=1+c_1+c_2+t \leq \frac{\Delta+7c_1+8c_2+2c_3+6t+1}{6}\leq \frac{m+1}{6}$. The equality holds if and only if $\iota(G,C_4)=1+c_1+c_2+t$, $1+c_1+c_2+t=\frac{\Delta+7c_1+8c_2+2c_3+6t+1}{6}$ and $m=\Delta+7c_1+8c_2+2c_3+6t$. We shall deduce $c_1=1$, $c_2=c_3=0$, $t=0$, $\Delta=4$, $e=3$ and $G[N[v]] \cong K_{1,4}$.

Let $H^*$ be the unique component of $\mathcal{H}_1=\mathcal{H}_s$. Then, $H^* \cong C_4$ and $e=e(N[v],V(H^*))=3$. Hence, $G$ is the graph obtained from $H^*$ and $G[N[v]]$ by adding three edges between the leaves of $G[N[v]]$ and some vertices of $H^*$. It is easy to check that there exists a vertex $u$ of $H^*$ such that $e(N_{H^*}[u],N[v]) \geq 2$, and thus, $\{u\}$ is a $C_4$-isolating set of $G$. We have $\iota(G,C_4) \leq |\{u\}|=1<\frac{4+3+4+1}{6}=\frac{m+1}{6}$.

(ii) It follows that $\mathcal{H}' \neq \emptyset$, that is, $c' \geq 1$. Recall that $S=N[v] \cup \bigcup_{H \in \mathcal{H}_s}V(H)$ and $G-S=G'' \ncong (\emptyset,\emptyset)$. For each $H \in \mathcal{H}_3$, let $D_H$ and $t(H)$ be defined as in (i). Also, let $t=\sum_{H \in \mathcal{H}_{3}}t(H)$, and we have $\sum_{H \in \mathcal{H}_{3}}|E(H)|=6t-c_3$. Therefore,
\begin{equation}
	\begin{array}{lllll}
		|E(G'')|&=m-|E(G[S])|-e(S,V\setminus S)\\
		&=m-|E(G[N[v]])|-\sum_{H \in \mathcal{H}_{s}}(|E(H)|+e(V(H),N[v]))\\
		&~~~~~~-\sum_{H \in \mathcal{H}'}e(V(H),N[v])\\
		&\leq m-\Delta-7c_1-8c_2-(6t-c_3)-3c_3-c'\\
		&=m-\Delta-7c_1-8c_2-2c_3-6t-c'.
	\end{array}
\end{equation}

For each $H \in \mathcal{H}_1 \cup \mathcal{H}_2$, let $u_H$ be defined as in (i). It is easy to see that $D_S=\{v\}\cup \bigcup_{H \in \mathcal{H}_1 \cup \mathcal{H}_2}\{u_{H}\}\cup \bigcup_{H \in \mathcal{H}_3}D_H$ is a $C_4$-isolating set of $G[S]$, and there exist no edges between the vertices of each component of $G[S]-N[D_S]$ and $V(G'')$. By (3), Lemmas \ref{lem2.1} and \ref{lem2.2}, and the induction hypothesis, we have
\begin{equation}
	\begin{array}{lllll}
		\iota(G,C_4)&\leq |D_S|+\iota(G'',C_4)\\
		&= |\{v\}\cup \bigcup_{H \in \mathcal{H}_1 \cup \mathcal{H}_2}\{u_{H}\}\cup \bigcup_{H \in \mathcal{H}_3}D_H|+\sum_{H \in \mathcal{H}'}\iota(H,C_4)\\
		&\leq 1+c_1+c_2+t+\sum_{H \in \mathcal{H}'}\frac{|E(H)|}{6}\\
		&=1+c_1+c_2+t+\frac{|E(G'')|}{6}\\
		&\leq 1+c_1+c_2+t+\frac{m-\Delta-7c_1-8c_2-2c_3-6t-c'}{6}\\
		&=\frac{m+1}{6}+\frac{5-\Delta-c_1-2c_2-2c_3-c'}{6}.
	\end{array}
\end{equation}
Since $\Delta \geq 4$, $c_1+c_2+c_3 \geq 1$ and $c' \geq 1$, $5-\Delta-c_1-2c_2-2c_3-c'<0$. Therefore, $\iota(G,C_4) < \frac{m+1}{6}$ by (4). This completes the proof of Theorem \ref{th1.4}.


\begin{thebibliography}{}
	
	\bibitem{Bollobas1998} B. Bollob\'{a}s, Modern Graph Theory, Springer, 1998.
	
	\bibitem{Bondy2008} J.A. Bondy, U.S.R. Murty, Graph Theory, GTM 244, Springer, 2008.
	
	\bibitem{Bartolo2023+} K. Bartolo, P. Borg, D. Scicluna, Isolation of squares in graphs, arXiv:2310.09128 [math.CO].
	
	\bibitem{Borg2020} P. Borg, Isolation of Cycles, Graphs Combin. 36 (2020) 631-637.
	
	\bibitem{Borg2023} P. Borg, Isolation of connected graphs, Discrete Appl. Math. 339 (2023) 154-165.
	
	\bibitem{Borg&Fenech2020} P. Borg, K. Fenech, P. Kaemawichanurat, Isolation of $k$-cliques, Discrete Math. 343 (2020) 111879.
	
	\bibitem{Borg&Fenech2022} P. Borg, K. Fenech, P. Kaemawichanurat, Isolation of $k$-cliques II, Discrete Math. 345 (2022) 112641.
	
	\bibitem{Borg&Kaemawichanurat2020} P. Borg, P. Kaemawichanurat, Partial domination of maximal outerplanar graphs, Discrete Appl. Math. 283 (2020) 306-314.
	
	\bibitem{Borg&Kaemawichanurat2023} P. Borg, P. Kaemawichanurat, Extensions of the Art Gallery Theorem, Ann. Comb. 27 (2023) 31-50.
	
	\bibitem{Caro2017} Y. Caro, A. Hansberg, Partial domination - the isolation number of a graph, Filomat 31 (2017) 3925-3944.
	
	\bibitem{Chen2023} J. Chen, S. Xu, $P_5$-isolation in graphs, Discrete Appl. Math. 340 (2023) 331-349.
	
	\bibitem{Favaron2021} O. Favaron, P. Kaemawichanurat, Inequalities between the $K_k$-isolation number and the independent $K_k$-isolation number of a graph, Discrete Appl. Math. 289 (2021) 93-97.
	
	\bibitem{Huang2023+} Y. Huang, G. Zhang, X. Jin, New results on the 1-isolation number of graphs without short cycles, arXiv:2308.00581 [math.CO].
	
	\bibitem{Kaemawichanurat2023} P. Kaemawichanurat, O. Favaron, Partial domination and irredundance numbers in graphs, Appl. Math. Comput. 457 (2023) 128-153.
	
	\bibitem{Yan2022} J. Yan, Isolation of the diamond graph, Bull. Malays. Math. Sci. Soc. 45 (2022) 1169-1181.
	
	\bibitem{Zhang2021} G. Zhang, B. Wu, $K_{1,2}$-isolation in graphs, Discrete Appl. Math. 304 (2021) 365-374.
	
	\bibitem{Zhang2022} G. Zhang, B. Wu, Isolation of cycles and trees in graphs, J. Xinjiang Univ. (Nat. Sci. Ed. Chin. Eng.) 39 (2022) 169-175.
	
\end{thebibliography}
\end{document}